\documentclass[12pt, a4paper]{amsart}
\usepackage[utf8]{inputenc}
\usepackage{amsthm}
\usepackage{amssymb}
\usepackage[shortlabels]{enumitem}
\setlist[enumerate]{topsep = 0mm, label = \normalfont(\roman*), listparindent=0mm, parsep=1mm}

\usepackage{xcolor, soul}
\definecolor{darkblue}{rgb}{0,0,0.6}
\usepackage[breaklinks, pdftex, ocgcolorlinks,colorlinks=true, citecolor=darkblue, filecolor=darkblue, linkcolor=darkblue, urlcolor=violet, linktocpage=true]{hyperref}

\usepackage{etoolbox}
\patchcmd{\thebibliography}{\leftmargin\labelwidth }{\leftmargin\labelwidth\labelsep=20pt}{}{}

\newtheorem{thm}{Theorem}[section]
\newtheorem{lemma}[thm]{Lemma}

\newtheorem{cor}[thm]{Corollary}
\theoremstyle{definition}
\newtheorem{defn}[thm]{Definition}

\newtheorem{qu}[thm]{Question}

\DeclareMathOperator{\Aut}{\operatorname{Aut}}

\title{The Wreath Product of Semiprime Skew Braces is Semiprime}

\author{Patrick I. Kinnear} 
\address{School of Mathematics, University of Edinburgh, James Clerk Maxwell Building, Peter Guthrie Tait Rd, Edinburgh EH9 3FD}
\email{P.Kinnear@sms.ed.ac.uk}

\keywords{Braces, Skew braces, Yang-Baxter equation, Semiprime}

\subjclass[2010]{Primary: 16T20. Secondary: 81R50}

\begin{document}

\maketitle

\begin{abstract}
In this note, we show that the wreath product of two semiprime skew braces is also a semiprime skew brace.
\end{abstract}

\section{Introduction}
The Yang-Baxter equation (YBE) is an equation of importance to several areas of research in mathematics and physics. An overview of the YBE can be found in Sections 1 and 2 of \cite{nichita2012introduction}. An algebraic structure known as a skew brace was introduced in \cite{guarnieri2017skew} to study a particular class of solutions of the YBE, called non-degenerate set-theoretic solutions. The results of \cite{bachiller2018solutions} fully reduce the classification of these solutions to the classification of skew braces, giving strong motivation for the study of skew braces and their structure.

Skew braces are a generalisation of Jacobson radical rings. As such, many ring-theoretic notions have useful analogues in the context of skew braces. The notion of an ideal has been defined for skew braces, and from this we may define several ring-theoretic properties such as being simple, nil, nilpotent, prime and semiprime. A detailed investigation of such properties of skew braces was initiated in \cite{konovalov2018skew}, and several questions were posed regarding prime and semiprime skew braces.

It is known that the property of being right nilpotent is preserved by wreath products of left braces (see Theorem 8.7 of \cite{gatevaivanova2009multipermutation} and Lemma 22 of \cite{smoktunowicz2018note}). In this note we will show that the wreath product of two semiprime skew braces is also a semiprime skew brace (Theorem \ref{t:wreath-semi}). This gives us one way to obtain new semiprime skew braces from old.

This result is proved in Section \ref{s:main}, and builds on the result that the semidirect product of semiprime skew braces is semiprime, which is established following some preliminary definitions in Section \ref{s:prelim}.

\section{Preliminaries}
\label{s:prelim}
In this section we give some key definitions and basic results, beginning with the definiton of a skew brace.

\begin{defn}
\label{d:skew-brace}
A \emph{skew brace} is a triple $(A, +, \circ)$ such that $(A, +)$ and $(A, \circ)$ are groups and the relation
\[
a \circ (b + c) = a \circ b - a + a \circ c
\]
holds for all $a, b, c \in A$.
\end{defn}

We denote the identities of $(A, +)$ and $(A, \circ)$ by $0$ and $1$ respectively. It can be checked that in any skew brace, $0 = 1$. For $a \in A$, the inverse in $(A, \circ)$ is written $a^{-1}$. Note that neither $(A, +)$ nor $(A, \circ)$ need be abelian: in the case where $(A, +)$ is abelian, then we have a structure called a left brace as first introduced by Rump in \cite{rump2007braces}. We remark that the morphisms for skew braces are the natural ones, and that a sub skew brace is simply a subset which is a subgroup with resepct to both operations $+$ and $\circ$.

It is reasonable to consider how we might obtain new skew braces from old. In this note we are concerned with the wreath product of skew braces, which is an instance of a semidirect product. This construction was defined for left braces in \cite{rump2008semidirect}, and generalised to skew braces in Corollary 3.37 of \cite{smoktunowicz2018skew}.

\begin{defn}
Let $(G, +, \circ), (H, + , \circ)$ be skew braces and $\sigma \colon H \to \Aut_{\text{Sk}}(G)$ be a group homomorphism from $(H, \circ)$ to the group of skew brace automorphisms of $G$. Then the \emph{semidirect product} $G \rtimes H$ of $G$ and $H$ \emph{via} $\sigma$ is the set $G \times H$ equipped with the following addition
\[
(g_1, h_1) + (g_2, h_2) = (g_1 + g_2, h_1 + h_2)
\]
and the following circle operation
\[
(g_1, h_1) \circ (g_2, h_2) = (g_1 \circ \sigma(h_1)(g_2), h_1 \circ h_2)
\]
for all $(g_1, h_1),(g_2, h_2) \in G \times H$.
\end{defn}

The wreath product of left braces was investigated in Corollaries 3.5 and 3.6 of \cite{cedo2010involutive}, and the construction was generalised to skew braces in Corollary 3.39 of \cite{smoktunowicz2018skew}.

\begin{defn}
\label{d:wreath-product}
Let $G, H$ be skew braces and consider the set
\[
W = \{f \colon H \to G \text{ such that } |\{h \in H : f(h) \neq 1\}| < \infty \}
\]
as a skew brace when we define addition and circle operations as
\begin{align*}
(f_1 + f_2)(h) &= f_1(h) + f_2(h)\\
(f_1 \circ f_2)(h) &= f_1(h) \circ f_2(h)
\end{align*}
for $f_1, f_2 \in W, h \in H$. Then the \emph{wreath product} of $G$ and $H$ is the skew brace $W \rtimes H$, with the action of $H$ on $W$ given by $\sigma \colon H \to \Aut_{\text{Sk}}(W)$ defined by $\sigma(h)(f)(x) = f(h \circ x)$ for all $x, h \in H, f \in W$. The wreath product of $G$ and $H$ is denoted $G \wr H$.
\end{defn}

Given a skew brace $A$ and an element $a \in A$, we can consider the map $\lambda_a \colon A \to A$ given by $\lambda_a (b) = -a + a \circ b$. Using this map, we define the notion of an ideal in a skew brace.

\begin{defn}
An \emph{ideal} of a skew brace $A$ is a normal subgroup $I$ of $(A, \circ)$ such that $\lambda_a(I) \subseteq I$ and $a + I = I + a$ for all $a \in A$.
\end{defn}

As noted in Remark 1.8 of \cite{guarnieri2017skew}, we can write $a \circ b = a + \lambda_a(b)$ for $a, b \in A$, from which we see that for $I$ an ideal of $A$ then $a \circ I = a + \lambda_a(I) \subseteq a + I$. Similarly, using the relation of Definition \ref{d:skew-brace}, we can write
\[
a + b = a \circ a^{-1} \circ (a + b) = a \circ (-a^{-1} + a^{-1} \circ b) = a \circ \lambda_{a^{-1}}(b).
\]
As above we see that $a + I \subseteq a \circ I$, so in fact $a \circ I = a + I$ where $I$ is an ideal of $A$ and $a \in A$. Moreover, as in Lemma 2.3 of \cite{guarnieri2017skew}, $I$ is also a normal subgroup of $(A, +)$. Since the cosets of $I$ with respect to $\circ$ and $+$ coincide, we may define quotient skew braces as follows.

\begin{defn}
Let $A$ be a skew brace and $I \subseteq A$ an ideal of $A$. Then the set of cosets of $I$ in $A$ can be made into a skew brace $A/I$ with operations given by
\begin{align*}
    (a + I) + (b + I) &= (a + b) + I\\
    (a + I) \circ (b + I) &= (a \circ b) + I
\end{align*}
for $a, b \in A$.
\end{defn}

Given a skew brace $(A, +, \circ)$, we may define an operation $*$ on $A$ given by $a*b = \lambda_a(b) - b$ for $a, b \in A$. For $B, C \subseteq A$, we write $B * C$ for the sub skew brace generated by (that is, the minimal sub skew brace containing) the set $\{b*c \colon b \in B, c \in C \}$. In the case that $A$ is a Jacobson radical ring, $*$ is simply the ring multiplication. The operation $*$ allows us to define analogues of many ring-theoretic concepts: in particular, it allows us to define semiprime skew braces as in \cite{konovalov2018skew}.

\begin{defn}
A skew brace $A$ is said to be \emph{semiprime} if the only ideal $I$ of $A$ which has $I * I = 0$ is the zero ideal.
\end{defn}

A skew brace $A$ is called trivial if $A * A = 0$, and such skew braces always yield the trivial solution to the YBE. A semiprime skew brace is a skew brace in which the only trivial ideal is the zero ideal, and so all its nonzero ideals yield non-degenerate set-theoretic solutions of the YBE which are nontrivial.

We will finish this section by showing that the semidirect product of semiprime skew braces is semiprime. This follows as an easy corollary of the following result, which gives sufficient conditions for a skew brace to be semiprime and has an analogous proof to the corresponding ring-theoretic result.

\begin{lemma}
\label{l:semiprime-ses}
Let $A$ be a skew brace and $I \subseteq A$ an ideal of $A$. If $I$ and $A/I$ are semiprime skew braces, then so is $A$.
\end{lemma}

\begin{proof}
Let $J$ be an ideal of $A$. Then $J +  I$ is an ideal of $A/I$ and it is straightforward to check that $(J + I)*(J + I) \subseteq J*J + I$. If $J*J = 0$ then it follows from semiprimality of $A/I$ that $J + I = 0$ in $A/I$, so $J \subseteq I$. But then $J$ is an ideal of $I$, so $J*J = 0$ implies $J=0$. Therefore $A$ is semiprime.  
\end{proof}

\begin{cor}
\label{c:semidir}
Let $G$ and $H$ be semiprime skew braces. Then the semidirect product $G \rtimes H$ via $\sigma \colon H \to \Aut_{\text{Sk}} (G)$, is a semiprime skew brace.
\end{cor}

\begin{proof}
It is known that $G \rtimes H$ is a skew brace. It is easily checked that $(G, 0) \cong G$ is an ideal of $G \rtimes H$, and that $(G \rtimes H)/(G, 0) \cong H$ as skew braces. The corollary then follows from Lemma \ref{l:semiprime-ses}.
\end{proof}

\section{The Wreath Product of Semiprime Skew Braces}
\label{s:main}
We will now extend Corollary \ref{c:semidir} to a result about the wreath product of semiprime skew braces. Given semiprime skew braces $G, H$, we begin by relating ideals in the skew brace $W$ of Definition \ref{d:wreath-product}, to ideals in $G$. This allows us to show in Lemma \ref{l:W-semi} that $W$ is semiprime. This, combined with Corollary \ref{c:semidir} yields Theorem \ref{t:wreath-semi} which states that $G \wr H$ is semiprime.

The lemma relating ideals of $G$ to ideals of $W$ depends on the easily verified fact that if $h \in H$, $g \in G$ and $f \in W$ are such that $g = f(h)$, then $g^{-1} = f^{-1}(h)$, for $f^{-1} \colon H \to G$ the inverse of $f$ in $(W, \circ)$.

\begin{lemma}
\label{l:W-proj}
Let $G, H$ be skew braces and $W$ be as in Definition \ref{d:wreath-product}, and $I$ an ideal of $W$. Then for any $h \in H$, the set
\[
\rho_{h}(I) = \{ g \in G : f(h) = g \text{ for some } f \in I \}
\]
is an ideal of $G$.
\end{lemma}

\begin{proof}
For convenience let $R = \rho_h(I)$. We begin by showing that $(R, \circ)$ is a subgroup of $(G, \circ)$. Note that the identity in $W$ is the function $0_W \colon H \to G$ such that $0_W(k) = 1_G$ for all $k\in H$. Then it is clear that $1_G \in R$.

Let $g_1, g_2 \in R$, so $g_1 = f_1(h), g_2 = f_2(h)$ for $f_1, f_2 \in I$. Then, as $I$ is an ideal of $W$, there is a function $f_1^{-1} \in I$, and we have $f_1^{-1}(h) = g_1^{-1} \in R$. It is easy to see that $g_1 \circ g_2 \in R$, and so we have that $(R, \circ)$ is a subgroup of $(G, \circ)$.

Now we show that $(R, \circ)$ is normal in $(G, \circ)$. Let $x \in G$, and define $\alpha_x \colon H \to G$ by
\[
\alpha_x(k) = \begin{cases}
x \text{ if } k = h\\
1_G \text{ otherwise}
\end{cases}
\]
so $\alpha_x \in W$. Then
\[
    x \circ g_1 \circ x^{-1} = \alpha_x(h) \circ f_1(h) \circ \alpha_x^{-1}(h) = (\alpha_x \circ f_1 \circ \alpha_x^{-1})(h) \in R
\]
since $(I, \circ)$ is a normal subgroup of $(W, \circ)$. This shows that $(R, \circ)$ is a normal subgroup of $(G, \circ)$.

In a similar way, we can easily show that $\lambda_a(R) \subseteq R$ and $a + R = R + a$, for any $a \in G$.  Therefore $R = \rho_{h}(I)$ is an ideal of $G$.
\end{proof}

Having associated ideals in $G$ to those in $W$, we may now observe the following regarding the skew brace $W$.

\begin{lemma}
\label{l:W-semi}
Let $G$ and $H$ be skew braces. If $G$ is semiprime, then the skew brace
\[
W = \{f \colon H \to G \text{ such that } |\{h \in H : f(h) \neq 1\}| < \infty \}
\]
is semiprime.
\end{lemma}

\begin{proof}
Let $I$ be an ideal of $W$ such that $I*I = 0$. Then for any $h \in H$ we may associate to $I$ the set $\rho_{h}(I)$ as in Lemma \ref{l:W-proj}, which by this lemma is an ideal of $G$.

Let $h \in H$ be fixed and consider $g_1, g_2 \in \rho_{h}(I)$, so $g_1 = f_1(h), g_2 = f_2(h)$ for some $f_1, f_2 \in I$. Then
\[
g_1 * g_2 = f_1(h) * f_2(h) = (f_1 * f_2)(h) = 0_W(h) = 1_G = 0_G
\]
which shows that $\rho_{h}(I) * \rho_{h}(I) = 0$. Since $G$ is semiprime, and $\rho_{h}(I)$ an ideal of $G$, we have that $\rho_{h}(I) = 0$.

Since $\rho_{h}(I) = 0$ for all $h \in H$, it follows that $I = 0$. This shows that $W$ is a semiprime skew brace.
\end{proof}

We now combine Corollary \ref{c:semidir} and Lemma \ref{l:W-semi} to give our final result concerning the wreath product of skew braces.

\begin{thm}
\label{t:wreath-semi}
If $G, H$ are semiprime skew braces, then their wreath product $G \wr H$ is also a semiprime skew brace.
\end{thm}

\begin{proof}
Let $G, H$ be semiprime skew braces. Then by Lemma \ref{l:W-semi}, we have that the skew brace 
\[
W = \{f \colon H \to G \text{ such that } |\{h \in H : f(h) \neq 1\}| < \infty \}
\]
is semiprime. Then by Corollary \ref{c:semidir}, the semidirect product $W \rtimes H$ is a semiprime skew brace. Then by definition, $G \wr H = W \rtimes H$ is a semiprime skew brace.
\end{proof}

We remark that the converse of Lemma \ref{l:W-semi} can be shown to hold, using a result similar to Lemma \ref{l:W-proj} to associate ideals in $W$ to ideals in $H$. Then, if the converse of Corollary \ref{c:semidir} were true, we would also be able to extend Theorem \ref{t:wreath-semi} to a biconditional statement. This leads us to ask: if the semidirect product of skew braces $G \rtimes H$ is semiprime, then are $G$ and $H$ also semiprime? It is clear that $H$ must be, but it is not known whether $G$ must be semiprime. This motivates the following question.

\begin{qu}
If $G, H$ are skew braces and their semidirect product $G \rtimes H$ via $\sigma$ is semiprime, is $G$ semiprime?
\end{qu}

\section{Acknowledgements}
The author would like to thank Agata Smoktunowicz for introducing him to the theory of braces, and for her feedback and encouragement in writing this note. The author also thanks Ivan Lau, Leandro Vendramin and the anonymous reviewers for their helpful comments on earlier versions of this manuscript.

\bibliographystyle{tf_USMAA}
\bibliography{main.bib}

\begin{thebibliography}{10}
\providecommand{\url}[1]{\texttt{#1}}
\providecommand{\urlprefix}{URL }
\expandafter\ifx\csname urlstyle\endcsname\relax
  \providecommand{\doi}[1]{DOI: \discretionary{}{}{}#1}\else
  \providecommand{\doi}{DOI: \discretionary{}{}{}\begingroup
  \urlstyle{rm}\Url}\fi
\providecommand{\eprint}[2][]{\url{#2}}

\bibitem{bachiller2018solutions}
Bachiller, D. (2018).
\newblock {Solutions of the {Y}ang-{B}axter equation associated to skew left
  braces, with applications to racks}.
\newblock \emph{J. Knot Theory Ramif.} 27(08): 1850055

\bibitem{cedo2010involutive}
Ced{\'o}, F., Jespers, E., {Del Rio}, A. (2010).
\newblock Involutive {{Y}ang-{B}axter} groups.
\newblock \emph{Trans. Amer. Math. Soc.} 362(5): 2541--2558.
\newblock \doi{10.1090/S0002-9947-09-04927-7}

\bibitem{gatevaivanova2009multipermutation}
Gateva-Ivanova, T., Cameron, P. (2009).
\newblock Multipermutation solutions of the {Y}ang--{B}axter equation.
\newblock arXiv eprint.
\newblock \eprint{0907.4276}

\bibitem{guarnieri2017skew}
Guarnieri, L., Vendramin, L. (2017).
\newblock {Skew braces and the {Y}ang-{B}axter equation}.
\newblock \emph{Math. Comp.} 86(307): 2519--2534.
\newblock \doi{10.1090/mcom/3161}

\bibitem{konovalov2018skew}
Konovalov, A., Smoktunowicz, A., Vendramin, L. (2018).
\newblock On skew braces and their ideals.
\newblock \emph{Exp. Math.} \doi{10.1080/10586458.2018.1492476}

\bibitem{nichita2012introduction}
Nichita, F. (2012).
\newblock {Introduction to the {Y}ang-{B}axter equation with open problems}.
\newblock \emph{Axioms.} 1(1): 33--37

\bibitem{rump2007braces}
{Rump}, W. (2007).
\newblock {Braces, radical rings, and the quantum {Y}ang-{B}axter equation}.
\newblock \emph{J. Algebra.} 307(1): 153 -- 170.
\newblock \doi{10.1016/j.jalgebra.2006.03.040}

\bibitem{rump2008semidirect}
Rump, W. (2008).
\newblock {Semidirect products in algebraic logic and solutions of the quantum
  {Y}ang-{B}axter equation}.
\newblock \emph{J. Algebra Appl.} 7(04): 471--490.
\newblock \doi{10.1142/S0219498808002904}

\bibitem{smoktunowicz2018note}
Smoktunowicz, A. (2018).
\newblock A note on set-theoretic solutions of the {Y}ang--{B}axter equation.
\newblock \emph{J. Algebra.} 500: 3--18

\bibitem{smoktunowicz2018skew}
Smoktunowicz, A., Vendramin, L. (2018).
\newblock {On skew braces (with an appendix by N. Byott and L. Vendramin)}.
\newblock \emph{J. Comb. Algebra} 2(1): 47--86

\end{thebibliography}

\end{document}